\documentclass[10pt,oneside,a4paper]{amsart}
\usepackage{latexsym}
\usepackage{amsfonts,amsmath,amssymb,indentfirst}
\usepackage[english]{babel}
\usepackage[all]{xy}
\usepackage[pdftex]{hyperref}
\newcommand{\bdism}{\begin{displaymath}}
\newcommand{\edism}{\end{displaymath}}

\newcommand{\oo}{\mathcal{O}}

\DeclareMathOperator{\supp}{Supp}
\newtheorem{theorem}{Theorem}[section]
\newtheorem{proposition}[theorem]{Proposition}
\newtheorem{corollary}[theorem]{Corollary}
\newtheorem{lemma}[theorem]{Lemma}

\author{\scshape Gabriele Di Cerbo}
\title{\bf A cohomological interpretation of Bogomolov's instability}
\begin{document}
\pagestyle{headings}
\begin{abstract}
We give a new proof of Bogomolov's instability theorem. Furthermore we prove that it is equivalent to a statement which characterizes when the first cohomology group of a suitable divisor does not vanish.
\end{abstract}
\date{December 06, 2011}
\maketitle

\section{Introduction}
\pagenumbering{arabic}
In the theory of stable vector bundles on surfaces the following theorem, known as Bogomolov's instability theorem, plays a central role:
\begin{theorem}[Bogomolov]\label{bog}
Let $X$ be a smooth projective surface and $V$ be a rank 2 vector bundle on $X$. If $c_{1}(V)^{2}>4c_{2}(V)$ then $V$ is unstable.
\end{theorem}
For the original proof we refer to \cite{Bog}, see also \cite{Reid}. This theorem was later proved by quite different techniques in \cite{Gie} and \cite{Miy}. Furthermore Reider used Theorem \ref{bog} to study adjoint linear series on surfaces and to derive his famous theorem, \cite{Rei}. \\
The first cohomological proof of Reider's theorem was given by Sakai in \cite {Sak}. His proof uses ideas of Serrano \cite{Ser}, and generalizes Reider's theorem to normal surfaces. The key point in Sakai's proof is the following theorem.
\begin{theorem}[Sakai]\label{sak}
Let $D$ be a big divisor with $D^{2}>0$ on a smooth projective surface $X$. If $H^{1}(X,\oo_{X}(K_{X}+D))\neq 0$ then there exists an effective divisor $E$ such that 
\begin{enumerate}
\item $D-2E$ is big;
\item $(D-E)\cdot E\leq 0$.
\end{enumerate}
\end{theorem}
As shown in \cite{Sak} Theorem \ref{sak} can be easily derived from Theorem \ref{bog}. Moreover Sakai gave an alternative proof based on Miyaoka's vanishing theorem for the Zariski decomposition of a divisor. Later Ein and Lazarsfeld show how to apply the Kawamata-Viehweg vanishing theorem to prove a part of Reider's theorem in \cite{Ein}. Based on these new techniques Fern\'andez del Busto gave an elegant proof of Bogomolov's inequality which uses only the Kawamata-Viehweg theorem, see \cite{Fer}. For a survey on these results we refer to \cite{Laz1}.\\
On the other hand, Mumford shows that we can use Bogomolov's theorem for rank 2 vector bundles to give a short proof of a generalized Kodaira vanishing for surfaces, see \cite{Fri}. This vanishing theorem is a little less general than the theorem of Kawamata-Viehweg. These results suggest that there should be a connection between Bogomolov's instability and some vanishing theorem. \\
In this note we prove 
\begin{theorem}\label{teo}
Bogomolov's instability theorem is equivalent to Theorem \ref{sak}.
\end{theorem}
Furthermore, using Sakai's proof of Theorem \ref{sak}, one gets a new proof of Bogomolov's instability theorem which is entirely cohomological.\\ 
We now outline the proof of Theorem \ref{teo}. After twisting the vector bundle $V$ with a line bundle we can assume that $V$ has a global section. Using this section we have that the extension class of the vector bundle is a nontrivial since $V$ is locally free. The first step of our proof follows Fern\'andez del Busto's argument \cite{Fer}. At this point we follow a different strategy. The numerical condition of Bogomolov's inequality allows us to apply Theorem \ref{sak} and we show directly that the divisor $E$ gives the destabilizing subsheaf.

\section{Preliminaries}
For the convenience of the reader we sketch the proof of Theorem \ref{sak}.
\begin{proof}
Let $D=P+N$ be the Zariski decomposition of $D$ and write $N=\sum \alpha_{j}E_{j}$ with each $\alpha_{j}$ positive and rational. By Sakai's lemma, Example 9.4.12 in \cite{Laz}, we know that $H^{1}(X,\oo_{X}(K_{X}+D-\left\lfloor N\right\rfloor))=0$ so $\left\lfloor N\right\rfloor >0$. Consider the following sequence of divisors: 
\bdism
D_{0}=D-\left\lfloor N\right\rfloor, \dots, D_{k}=D_{k-1}+E_{j_{k}},\dots,D_{n}=D.
\edism
If $D_{k}\cdot E_{j_{k}}>0$ for any $k$, we get the vanishing of $H^{1}(X,\oo_{X}(K_{X}+D))$. Thus we can collect all the $E_{j_{k}}$'s with positive intersection to construct a sequence $D_{0},\dots,D_{k}$ such that $(D-D_{k})\cdot E_{j}\leq 0$ for all irreducible components $E_{j}$ of $D-D_{k}$. Now a computation shows that $E:=D-D_{k}$ is the required divisor. 
\end{proof}
\begin{corollary}\label{cor}
Let $D$ and $E$ be as above then 
\bdism
H^{1}(X,\oo_{X}(K_{X}+D-E))=0.
\edism
\end{corollary}
\begin{proof}
By the above construction
\bdism
H^{1}(X,K_{X}+D_{0})=H^{1}(X,K_{X}+D_{k}).
\edism
Since $D_{0}=D-\left\lfloor N\right\rfloor$ and $D_{k}=D-E$, the result follows from Sakai's lemma.
\end{proof}
In conclusion we recall two results which will be used in the proof of the main theorem. 
\begin{lemma}\label{lem}
Let $f:Y\rightarrow X$ be a birational morphism between smooth projective surfaces and $\widetilde{L}$ a divisor on $Y$. Set $L:=f_{*}\widetilde{L}$, if $\widetilde{L}^{2}>0$ and $L$ is big then $\widetilde{L}$ is big.
\end{lemma}
\begin{proof}
Lemma 3 in \cite{Sak}.
\end{proof}
\begin{proposition}\label{pro}
Let $f:\widetilde{X}\rightarrow X$ be a birational morphism between smooth projective surfaces. Let $\widetilde{D}$ be a divisor on $\widetilde{X}$ such that $\widetilde{D}^{2}>0$. Suppose there is a divisor $\widetilde{E}$ which satisfies the conclusions of Theorem \ref{sak} and let $D:=f_{*}\widetilde{D}$, $E:=f_{*}\widetilde{E}$ and $\alpha:=D^{2}-\widetilde{D}^{2}$. If $D$ is nef and $E$ effective we have
\bdism
0\leq D\cdot E <\alpha/2.
\edism
\begin{proof}
See Proposition 2 in \cite{Sak}.
\end{proof}
\end{proposition}

\section{Main Theorem}
We can now prove the main result of the paper.
\begin{proof}[Proof of Theorem \ref{teo}]
As mentioned before, Theorem \ref{sak} can be easily proved using Bogomolov's instability, see \cite{Sak} p. 307. \\
We now want to show that Theorem \ref{sak} implies Bogomolov's theorem. Since the inequality in Theorem \ref{bog} is invariant under twisting with a line bundle we can assume that $V$ is globally generated, $\det(V)$ is ample and $c_{2}(V)>0$. Taking a general section $s$ of $V$ we get the following exact sequence 
\bdism
0\rightarrow \oo_{X} \rightarrow V \rightarrow L\otimes I_{Z} \rightarrow 0, 
\edism
where $L:=\det(V)$ and $Z$ is the zero locus of $s$. Then we have $c_{2}(V)=|Z|$, the length of $Z$. \\
Since $V$ is locally free, the above extension is nontrivial and then
\bdism
H^{1}(X,\oo_{X}(K_{X}+L)\otimes I_{Z})\neq 0.
\edism
Let $\pi:Y\rightarrow X$ be the blow up of $X$ at all points in $Z$. Let $E_{j}$ be the exceptional curve over $x_{j}\in Z$, then
\bdism
H^{1}(Y,\oo_{Y}(K_{Y}+\pi^{*}L-2\sum_{j} E_{j}))=H^{1}(X,\oo_{X}(K_{X}+L)\otimes I_{Z})\neq 0.
\edism 
Define $\widetilde{L}:=\pi^{*}L-2\sum_{j} E_{j}$. Thus, we have
\bdism
\widetilde{L}^{2}=(\pi^{*}L)^{2}+4\sum_{j} E_{j}^{2}=c_{1}^{2}(V)-4c_{2}(V)>0
\edism
so $\widetilde{L}$ is big by Lemma \ref{lem}. \\
By applying Theorem \ref{sak} we get an effective divisor $\widetilde{E}_{s}$ such that
\begin{enumerate}
\item $\widetilde{L}-2\widetilde{E}_{s}$ is big;
\item $(\widetilde{L}-\widetilde{E}_{s})\cdot \widetilde{E}_{s}\leq 0$.
\end{enumerate}
Note that $\widetilde{E}_{s}$ depends on the section $s$ that we choose at the beginning. Let $E_{s}:=\pi_{*}\widetilde{E}_{s}$.
We want to show that, for any $s$, $E_{s}$ passes through at least one point of $Z$. Let $\widetilde{E}_{s}:=\pi^{*}E_{s}+\sum a_{i} E_{i}$, where $E_{i}$ are the exceptional divisors. It suffices to show that there exists an index $i$ such that $a_{i}<0$. Write $\widetilde{L}-\widetilde{E}_{s}=\pi^{*}W_{s}-\sum (a_{i}+2)E_{i}$ where $W_{s}:=L-E_{s}$. Thus by (2) we have 
\bdism
E_{s}\cdot W_{s}+\sum_{i}a_{i}(a_{i}+2)\leq 0.
\edism 
Then if we show that $E_{s}\cdot W_{s}>0$, we must have a negative $a_{i}$ and then $x_{i}\in \supp(E_{s})$. By construction $L=E_{s}+W_{s}$, $L\cdot E_{s}>0$ and 
\bdism
L\cdot W_{s}=(L-2E_{s})\cdot L+L\cdot E_{s}=(\widetilde{L}-2\widetilde{E}_{s})\cdot \pi^{*}L+L\cdot E_{s}> 0
\edism
by $(1)$. From the Hodge index theorem we get $E_{s}\cdot W_{s}>0$. \\
Now we need a result in \cite{Fer}, called the uniform multiplicity property. See also \cite{Laz1}.
\begin{lemma}
Choosing $s$ and $E_{s}$ generally we can assume that the multiplicity of $E_{s}$ at every point of $Z$ is the same.
\end{lemma}
Since for any $s$ exists $x\in Z$ such that $x\in \supp(E_{s})$, by the uniform multiplicity property, we can choose $s$ and $E_{s}$ generally such that $Z\subset \supp(E_{s})$. For simplicity, we denote this divisor by $E$. \\ 
$Z\subset \supp(E)$ implies that the multiplication by $E$ defines a map $\oo_{X}(L-E) \rightarrow \oo_{X}(L)\otimes I_{Z}$. Since the cohomology group in Corollary \ref{cor} vanishes this map lifts to an injective map $\oo_{X}(L-E) \rightarrow V$.
Thus, $\oo_{X}(L-E)$ is a subsheaf of $V$.\\
It remains to prove that $V$ is unstable. This is equivalent to showing: 
\bdism
(L-2E)^{2}>0, \quad \quad (L-2E)\cdot L>0.
\edism
For the first inequality we consider the following exact sequence
\bdism
0\rightarrow \oo_{X}(L-E) \rightarrow V \rightarrow \oo_{X}(E)\otimes I_{Z'} \rightarrow 0, 
\edism
for some zero dimensional scheme $Z'$. Then $c_{1}(V)=L$ and $c_{2}(V)=(L-E)\cdot E+|Z'|$ and by hypothesis we get 
\bdism
(L-2E)^{2}>4|Z'|>0.
\edism 
For the second one we note that
\bdism
\alpha=c_{1}^{2}(V)-c_{1}^{2}(V)+4c_{2}(V)=4c_{2}(V),
\edism
and Proposition \ref{pro} gives the following 
\bdism
L\cdot E< 2 c_{2}(V).
\edism
Then
\bdism
L^{2}>4c_{2}(V)>2L\cdot E. \qedhere
\edism
\end{proof}
\section*{Acknowledgments}
First I would like to express my gratitude to Professor J\'anos Koll\'ar for his constant support and many enlightening discussions. I also would like to thank Professor Robert Lazarsfeld, Luca Di Cerbo and the referee for constructive comments on the paper.

\noindent
{Princeton University, Princeton NJ 08544-1000}

\noindent{gdi@math.princeton.edu}

\end{document}